\documentclass{article}
\usepackage[utf8]{inputenc}
\usepackage[english]{babel}
\usepackage{amsfonts}
\usepackage{amsmath}
\usepackage{amssymb}
\usepackage{graphicx}
\usepackage[margin= 3cm]{geometry}
\newtheorem{theorem}{Theorem}

\newtheorem{corollary}[theorem]{Corollary}

\newtheorem{definition}{Definition}

\newtheorem{lemma}{Lemma}

\newtheorem{remark}[theorem]{Remark}

\newenvironment{proof}[1][Proof]{\noindent\textbf{#1.} }{\ \rule{0.5em}{0.5em}}
\usepackage[utf8]{inputenc}
\usepackage{verbatim}
\title{Cayley Graphs on Billiard Surfaces, and Their Genus}
\author{
  Joanna Grzegrzolka\\
  Lee University\\
  \texttt{jgrzeg00@leeu.edu} 
  \and
    Jaime Lynne McCartney\\
  Dalton State College\\
  \texttt{jmccart1@daltonstate.edu}
  \and
  Jason Schmurr\\
  Lee University\\
  \texttt{jschmurr@leeuniversity.edu}
}
\usepackage[
backend=biber,
style=alphabetic,
sorting=ynt
]{biblatex}
\begin{document}
\maketitle
\begin{abstract}
In this article we discuss a connection between two well-known constructions in mathematics: Cayley graphs and rational billiard surfaces. We describe a natural way to draw a Cayley graph of a dihedral group on each rational billiard surface. Both of these objects have the concept of ``genus'' attached to them. For the Cayley graph, the genus is defined to be the lowest genus amongst all surfaces that the graph can be drawn on without edge crossings. We prove that the genus of a Cayley graph associated with a triangular billiard table is always zero or one (Theorem \ref{main_theorem}). One reason this is interesting is that there exist triangular billiard surfaces of arbitrarily high genus, so the genus of the associated graph is often much lower than the genus of the billiard surface.

\end{abstract}
\section{Rational Billiard Surfaces}
The \emph{rational polygonal billiard surface} is a famous construction in topological dynamics. See \cite{Masur06rationalbilliards} for an excellent survey. Although billiard surfaces have served both as motivation for and examples of recent advances in sophisticated mathematics such as the work of Mirzhakani on moduli spaces, they have an intuitive construction. When following the path of a point mass bouncing around inside a polygon, when the point strikes a wall of the polygon we continue its path in a straight line in a reflected copy of the polygonal table. See Figure \ref{fig:433_process}.

\begin{figure}[h]
    \centering
    \includegraphics[scale=.15]{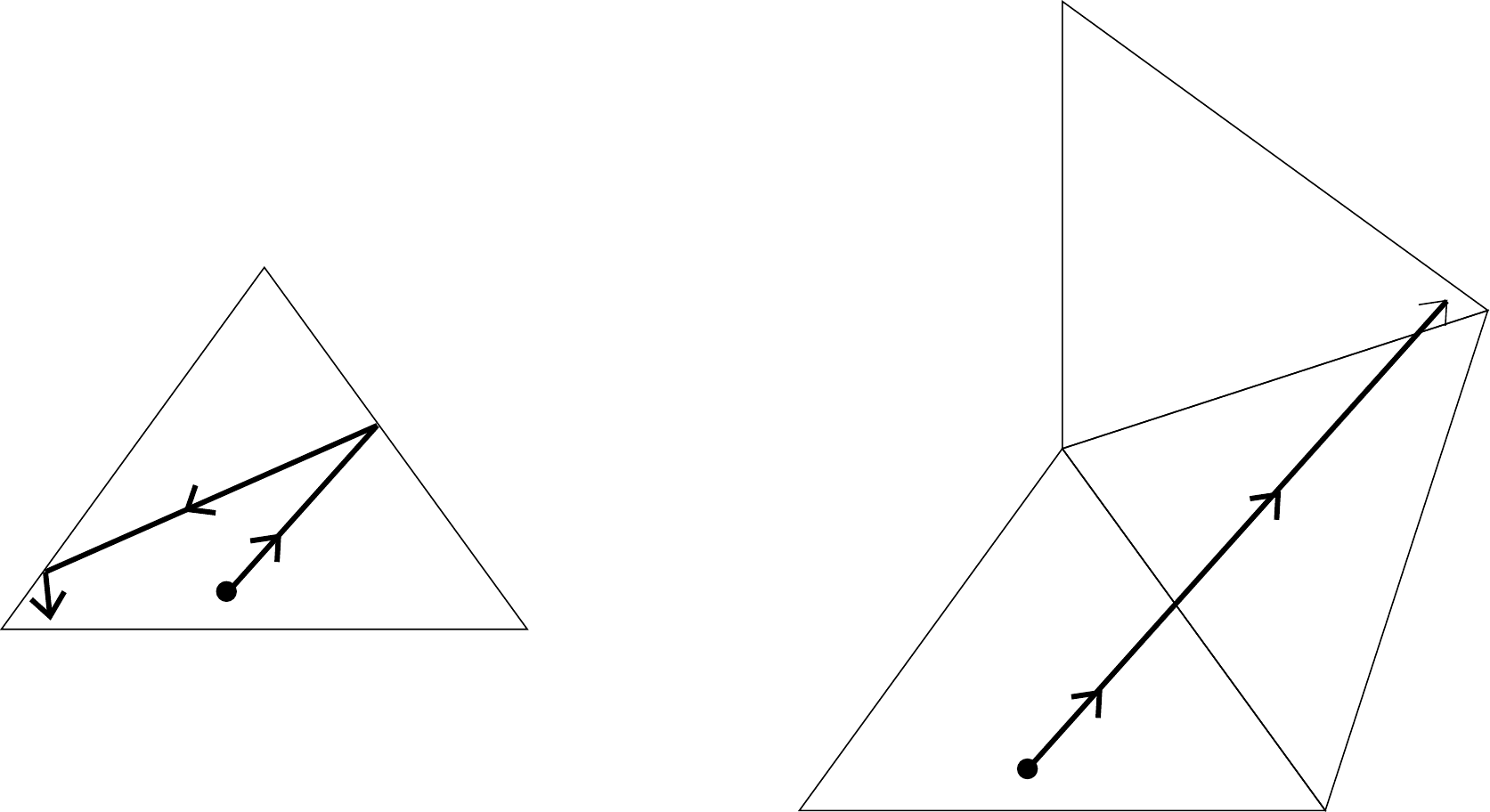}
    \caption{Unfolding a billiard path in a triangle with angles $\dfrac{3\pi}{10}$, $\dfrac{3\pi}{10}$, $\dfrac{4\pi}{10}$.}
    \label{fig:433_process}
\end{figure}

Copies of the table that correspond to identical directions of the billiard ball are identified. This is known as ``unfolding'' the path. In this way, a billiard path can be viewed as a straight line on a compact surface instead of as a collection of line segments of various slopes.

Here is a more rigorous description of a polygonal billiards surface. Begin with a polygonal region $R$ in the plane whose internal angles are rational multiples of $\pi$. Consider the set of reflections across the lines through the $m$ sides of $R$. Let $r_1,\ldots,r_m$ be the derivatives of these reflections -- such derivatives are reflections across lines through the origin. These derivatives generate a group $\Gamma$ under composition that consists of reflections across the origin (compositions of odd numbers of reflections) and rotations about the origin (compositions of even numbers of reflections). Suppose that the internal angles of $R$ are $\alpha_1,\ldots,\alpha_n$, where $\alpha_i=\dfrac{p_i\pi}{n}$ and $n,p_1,\ldots,p_m \in \mathbb{Z}$. Further, suppose that $n$ is the ``least common denominator'' in the sense that $\gcd(n,p_1,\ldots,p_m)=1$. It can be shown that in fact $\Gamma$ is $D_n$, the dihedral group of order $2n$\,.

Next, we consider the set $S=\{\sigma R :\sigma\in\Gamma\}$ of the $2n$ copies of $R$ transformed by elements of $\Gamma$.  We construct the rational billiards surface for $R$ by ``gluing'' two copies together along a side if one copy can be seen as a mirror image of the other, reflected across that side. Specifically, for each $\sigma_j\in\Gamma$ and each $r_i$, we glue together $\sigma_j r_i R$ and $\sigma_j R$ along the ``shared side''. This is because, if $e_i$ is the side of $R$ corresponding to $r_i$ then reflection across $\sigma(e_i)$ is equivalent to $\sigma r_i\sigma^{-1}$, and so the reflection across $\sigma(e_i)$ of $\sigma(R)$ is $\sigma r_i\sigma^{-1}(\sigma R)=\sigma r_i R$.

The result is a closed flat surface known as a \emph{translation surface} since, viewed as complex manifold with flat structure, its change-of-coordinate maps (away from a finite number of singular points located at the corners of the copies of $R$) are Euclidean translations.

In this paper we will focus on billiard surfaces arising from rational triangles. Let $T(p_1,p_2,p_3)$ denote a rational triangle with internal angles $\alpha_i=\dfrac{p_i\pi}{n}$, where $n=p_1+p_2+p_3$ and $\gcd(p_1,p_2,p_3)=1$. Let $X(p_1,p_2,p_3)$ denote the billiard surface arising from billiards in $T(p_1,p_2,p_3)$. See Figure \ref{fig:X334} for a diagram of $X(3,3,4)$. The numerical labels indicate side identifications. This construction has been described as far back as 1936 by Fox and Kershner in \cite{fox_kershner}.

\begin{figure}[!htb]
    \centering
    \begin{minipage}{.5\textwidth}
    \centering
    \includegraphics[scale=.15]{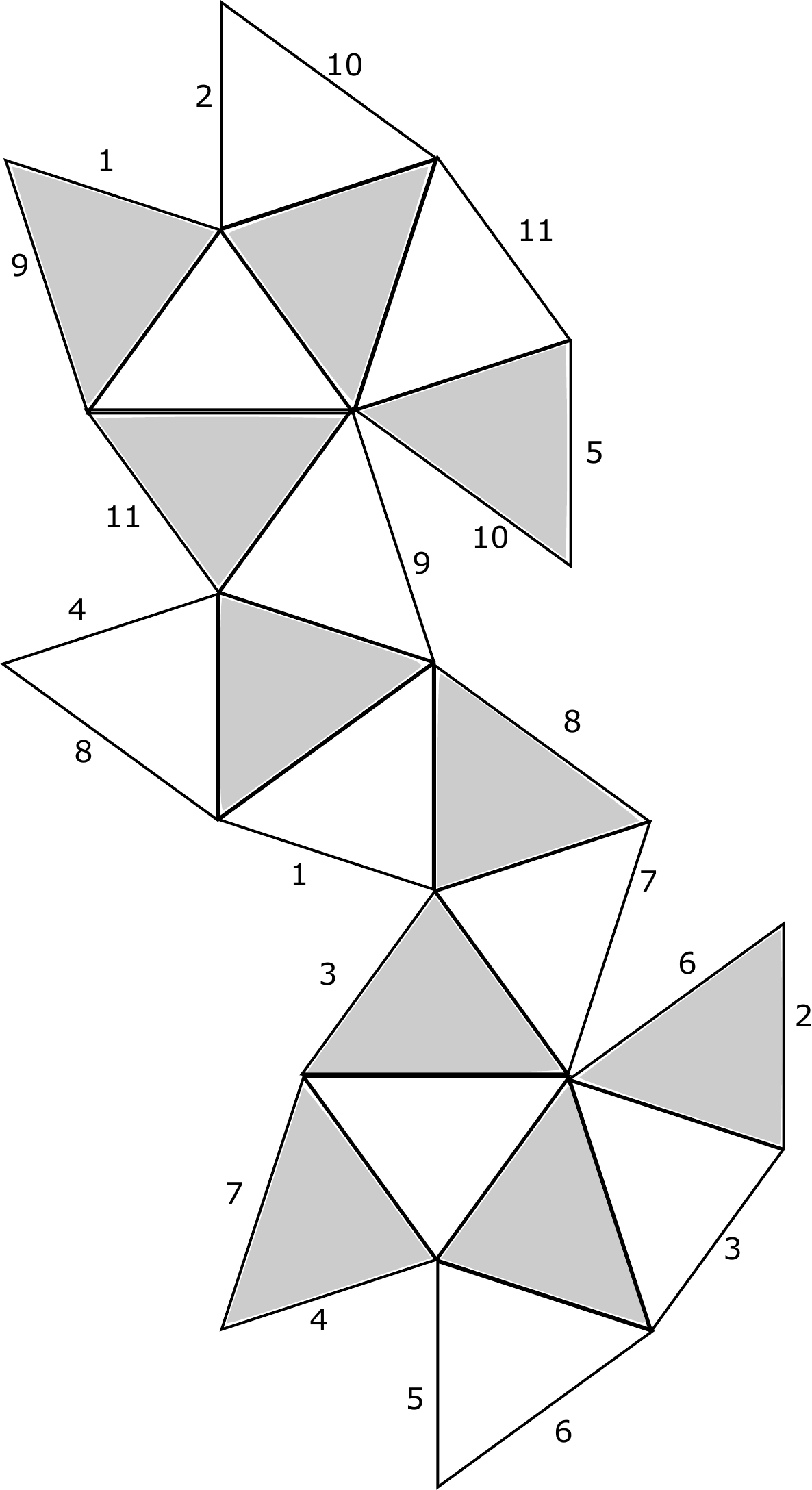}
    \caption{The billiard surface $X(3,3,4)$. }
    \label{fig:X334}

    \end{minipage}%
    \begin{minipage}{0.5\textwidth}
        \centering
    \includegraphics[scale=.18]{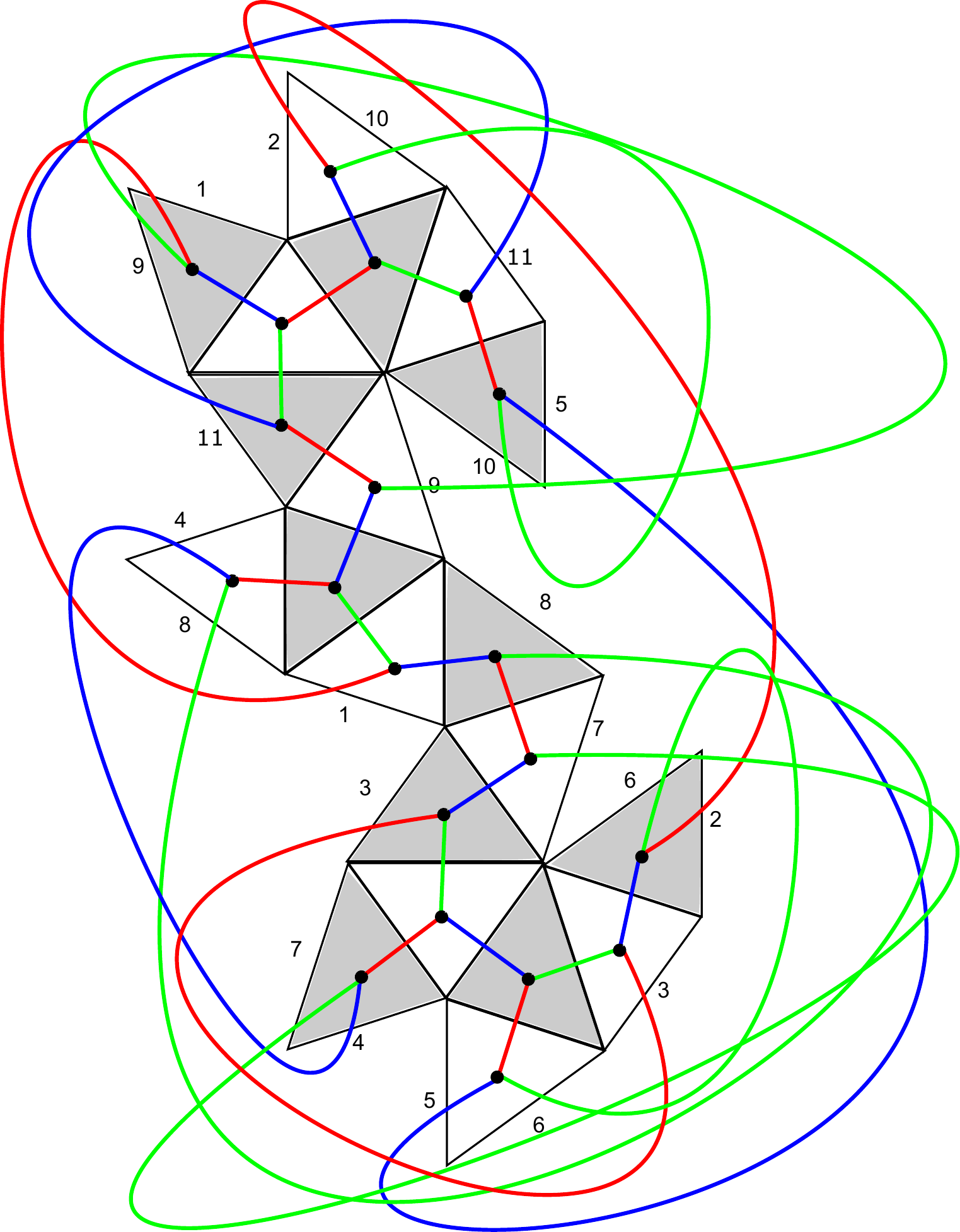}
    \caption{The graph $G(3,3,4)$ drawn on $X(3,3,4)$.}
    \label{fig:Gamma334}
    \end{minipage}
\end{figure}
\section{Cayley Graphs on Billiard Surfaces}

\subsection{Drawing the Graph}
Next, we will make a connection between billiard surfaces and graphs. We draw a graph $G(p_1,p_2,p_3)$ on a triangular billiard surface $X(p_1,p_2,p_3)$ by drawing a vertex in the center of each triangle and connecting two vertices with an edge if and only if the vertices lie inside triangles which share a side. See Figure \ref{fig:Gamma334}.
Because each copy of $T(p_1,p_2,p_3)$ in $X(p_1,p_2,p_3)$ is associated with a unique element of $D_n$, we can identify the vertices of the graph with elements of $D_n$. We choose one initial copy of $T(p_1,p_2,p_3)$ and denote it $eT$, where $e$ denotes the identity element of $D_n$. Let $a$, $b$, and $c$ denote reflections across lines parallel to sides opposite $\alpha_1$, $\alpha_2$, and $\alpha_3$, respectively. See Figure \ref{fig:reflections}. Then the triangle $eT$ is adjacent to triangles $aT$, $bT$, and $cT$, so we have edges connecting the vertex $e$ to the vertices $a$, $b$, and $c$. Observe that if $\sigma \in D_n$ then the reflection of $\sigma T$ across the side of $\sigma T$ opposite the $\alpha_1$ angle of $\sigma T$ is $\sigma a \sigma^{-1}(\sigma T)=\sigma a T$. This works the same way for the other two sides of the triangle. Hence, in $G(p_1,p_2,p_3)$, the vertex $\sigma$ is connected by edges to $\sigma a $, $\sigma b $, and $\sigma c $. 


In fact this graph is an example of a \emph{Cayley graph}. We shall review the necessary background for Cayley graphs in the next section.

\subsection{Notation and Basic Formulas}

\begin{figure}[h]
    \centering
    \includegraphics[scale=.15]{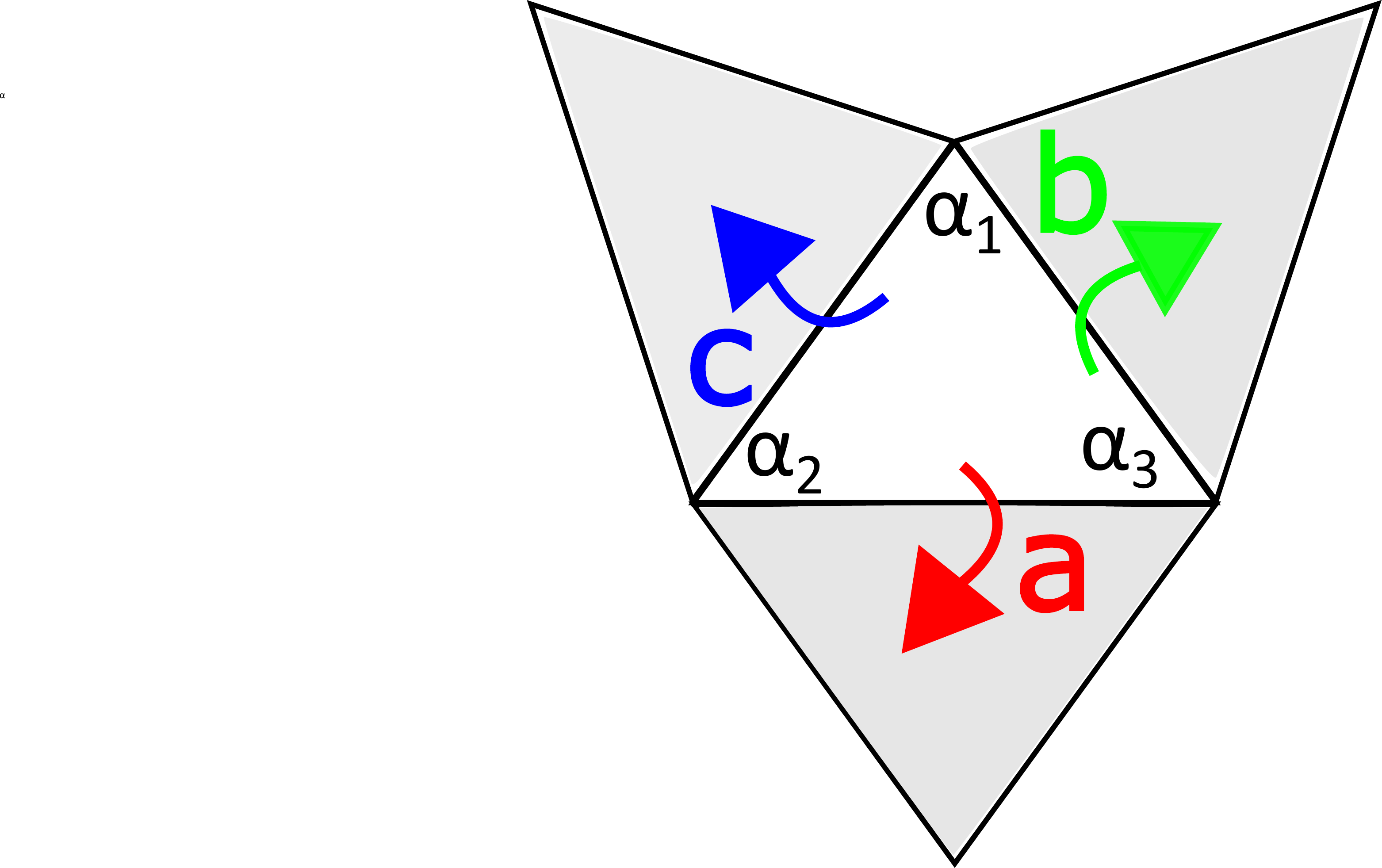}
    \caption{The reflections that generate $Cay(\{a,b,c\}, D_{n})$. }
    \label{fig:reflections}
\end{figure}

We write $Rot(\theta)$ to denote Euclidean rotation about the origin by $\theta$ and $Ref(\theta)$ to denote reflection across the line through the origin that makes an angle of $\theta$ with the positive horizontal axis. Then the following formulas hold:

$$Ref(\theta_1)Ref(\theta_2)=Rot\left(2[\theta_1-\theta_2]\right)$$
$$Ref(\theta_1)Rot(\theta_2)=Ref\left(\theta_1-\dfrac{1}{2}\theta_2\right)$$
$$Rot(\theta_1)Ref(\theta_2)=Ref\left(\theta_2+\dfrac{1}{2}\theta_1\right)$$

\begin{remark}
Suppose that our triangle is as in Figure \ref{fig:reflections}, oriented so that the base of the triangle is parallel to the horizontal axis. It follows from the elementary properties of Euclidean reflections and Euclidean rotations listed above that:

\begin{enumerate}
    \item $ab$ is rotation by $2\alpha_3$
    \item $ac$ is rotation by $-2\alpha_2$
    \item $aba$ is reflection across a line making angle $\alpha_3$ with the positive horizontal axis. 
\end{enumerate}
\end{remark}

\begin{remark}
There is a correspondence between closed billiard paths on $T(p_1,p_2,p_3)$, cylinders on $X(p_1,p_2,p_3)$ and circuits in $G(p_1,p_2,p_3)$. Specifically, every infinite family of parallel equal-length closed billiard paths in $T(p_1,p_2,p_3)$ unfolds to a cylinder on $X(p_1,p_2,p_3)$. The sequence of sides of $T$ struck by each path in this family determines a sequence of reflections in $\Gamma(p_1,p_2,p_3)$, which form a closed circuit in $G(p_1,p_2,p_3)$.
\end{remark}

For example, every right triangle has an infinite family of billiard paths consisting of striking four walls each; this corresponds to the closed circuit $abcb$ in the corresponding Cayley graph (see Lemma \ref{lemma:isosceles_four_word}). Similarly, the length six circuit identified in Theorem \ref{main_theorem} corresponds to the family of doubles of the well-known ``Fagnano orbit'' that exists as a closed billiard path in any acute triangle.

See Figure \ref{fig:X334} and Figure \ref{fig:Gamma334} for an example of a billiard surface and its accompanying Cayley graph. Although the genus of $X(3,3,4)$ is four, we shall show that the genus of $G(3,3,4)$ is zero.

\section{Graph Theory Background}
\subsection{Cayley Graphs}
Let $\Gamma$ be group. We say that a subset $S=\{g_1,g_2,...,g_n\}$ of $\Gamma$ is a \emph{generating set} for $\Gamma$ if every element of $\Gamma$ can be expressed as a product of elements of $S$ (and their inverses). We do not require $S$ to be minimal -- that is, we do not exclude the possibility that a proper subset of $S$ may also be a generating set of $\Gamma$. The \emph{Cayley graph} $Cay(S, \Gamma)$ of a group $\Gamma$ with generating set $\{g_1,g_2,\ldots,g_n\}$ is a graph whose vertices are the elements of $\Gamma$, and whose edges represent multiplication by an element of the generating set. We draw an edge from $x$ to $y$ if $y=xg_i$ for some generator $g_i$. For example, in Figure \ref{fig:K33} we see a drawing of the Cayley graph $Cay(\{a,b,c\},D_3)$, where $D_n$ is the dihedral group with $2n$ elements, and $\{a,b,c\}$ is the set of the three reflection elements in $D_3$. In graph theory, this graph is known as the \emph{complete bipartite graph} $K_{3,3}$. A graph is \emph{bipartite} if its vertex set can be partitioned into two subsets $V_1$ and $V_2$ such that each element of $V_1$ is adjacent only to elements of $V_2$. Note that in general a Cayley graph is a directed graph. However, because the generating elements we use in this paper all have order 2, we replace the pairs of oppositely directed edges with single undirected edges to form an undirected graph.  

\begin{figure}[h]
    \centering
    \includegraphics[scale=.3]{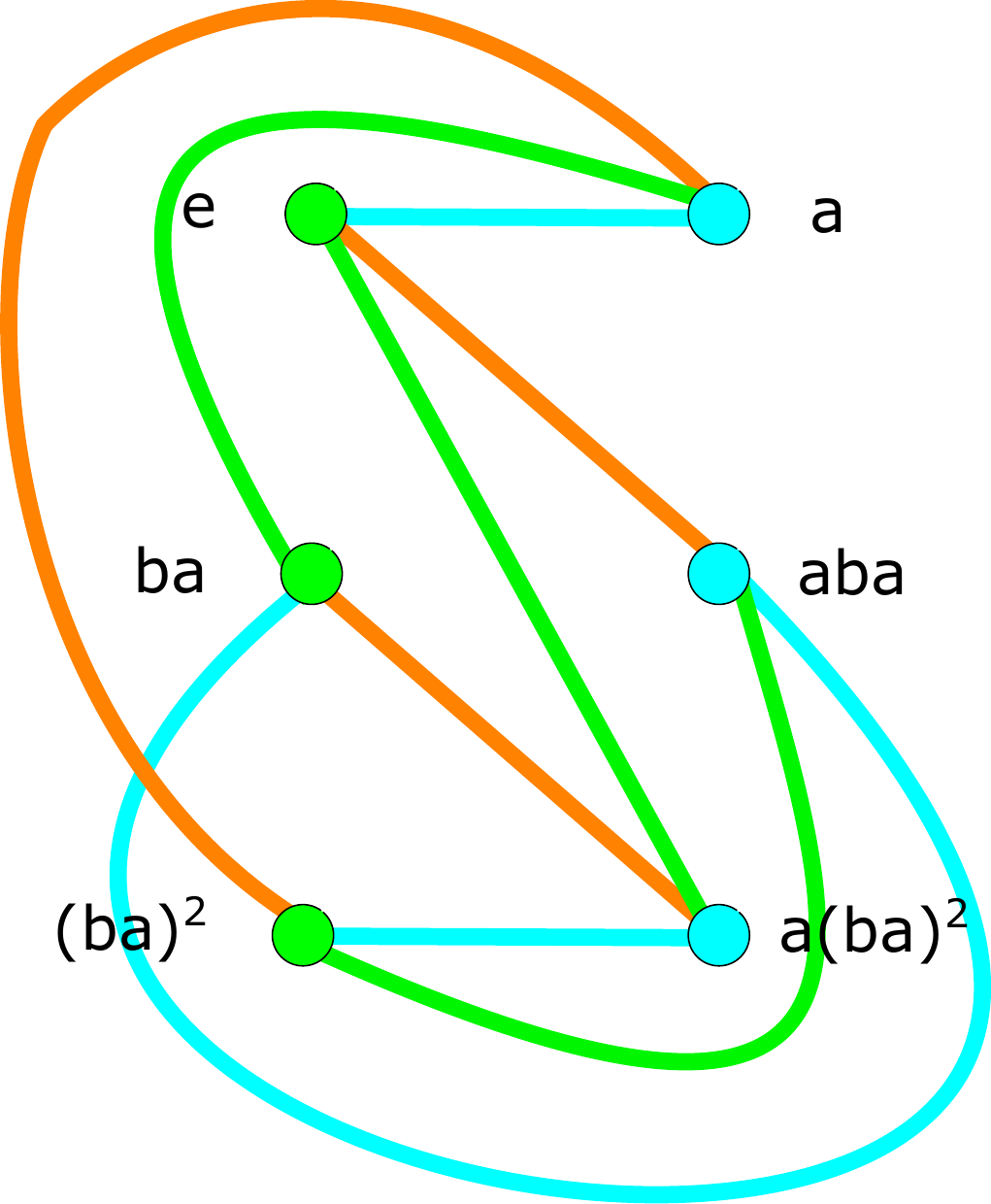}
    \caption{The Cayley graph for the (1,1,1) triangle is isomorphic to $K_{3,3}$.}
    \label{fig:K33}
\end{figure}
\subsection{Graph Genus and Graph Rotations}
\begin{definition}
The \emph{genus} of a graph is the smallest non-negative integer $g$ such that the graph can be drawn on a surface of genus $g$. 
\end{definition}

A graph of genus zero is called \emph{planar}. A classic theorem which can aid in computing graph genus is Kuratowksi's Theorem, which uses the concept of a \emph{subdivision} of a graph. A subdivision of a graph $G$ is a graph obtained from $G$ by adding vertices of order 2 to some of the edges of $G$. See Figure \ref{fig:K33subdivision} for an example of a graph which is a subdivision of $K_{3,3}$.

\begin{theorem}
(Kuratowski) A graph $G$ is not planar if and only if $G$ contains a subgraph that is a subdivision of either the complete graph $K_5$ or of the complete bipartite graph $K_{3,3}$.
\end{theorem}

Another tool for computing the genus of a graph is the concept of a \emph{graph rotation}. The following definitions are from \cite{hartsfield_ringel}. A \emph{rotation of a vertex} is an ordered cyclic listing of the vertices adjacent to that vertex. A \emph{graph rotation} consists of rotations of each vertex of the graph. This term is not to be confused with the concept of a Euclidean rotation of the plane.

A \emph{circuit} of a graph is a sequence of vertices $v_i$ and edges $E_i$ of the form $v_0E_1v_1E_2v_2\ldots E_nv_n$ such that $v_0=v_n$ and such that for each $i$, $E_i$ connects $v_{i-1}$ to $v_i$. We say that such a circuit has \emph{length} $n$. We may also choose to represent a circuit by listing only the subsequence consisting of the vertices, and omitting the final vertex since it is equal to the first: $v_0\ldots v_{n-1}$.
\begin{remark}
Observe that for any graph $G=Cay(\{a,b,c\},D_n)$ where $a$, $b$, and $c$ represent reflections:

\begin{enumerate}
    \item Every circuit $v_0E_1v_1E_2v_2\ldots E_nv_n$ corresponds to a relation $E_1E_2\ldots E_n=R_0$, where $R_0$ is the identity element of $D_n$.
    \item It follows from the previous observation that since each $E_i$ is a Euclidean reflection and $R_0$ is a Euclidean rotation, circuits in $G$ are always of even length. From a graph theoretic perspective, this is true because $G$ is bipartite: the reflections and Euclidean rotations form the two bipartite sets.
\end{enumerate}
\end{remark}

A graph rotation of a graph $G$ induces a set of circuits on $G$ such that each edge is traveled once in each direction. The circuits are obtained in the following way: the circuit which contains $...v_iv_j$ continues as $...v_iv_jv_k$, where $v_k$ is the vertex directly following $v_i$ in the rotation of $v_j$. For example, consider the following rotation of $K_{3,3}$:
\begin{align*}\\v_0.v_1v_3v_5\\v_1.v_0v_4v_2\\v_2.v_1v_5v_3\\v_3.v_0v_2v_4\\v_4.v_1v_3v_5\\v_5.v_0v_2v_4\\\end{align*}

The notation $w.v_iv_jv_k$ means that the vertex $w$ is adjacent to exactly the vertices $v_i,v_j,v_k$, and that the cyclic listing of these vertices is $v_iv_jv_k$.

The rotation in our example induces three circuits. This includes two circuits of length four:  $v_0v_3v_2v_1$ and $v_1v_2v_5v_4$; and one circuit of length eight: $v_0v_1v_4v_3v_0v_5v_2v_3v_4v_5$. 

Let $r(\rho)$ denote the number of circuits induced by a rotation $\rho$ of a graph $G$. So in the previous example, $r(\rho)=3$. We call $\rho$ a  \emph{maximal rotation} of $G$ if $r(\rho)=\max_{\rho_i}\{r\left(\rho_i\right)\}$ where the $\rho_i$ vary over all possible rotations of $G$.  

The following formula provides the connection between graph genus and graph rotations. It is related to Euler's characteristic formula, with circuits playing the role of the polygonal ``faces'' of the surface.

\begin{theorem}\label{thm:rotation_genus}\cite{hartsfield_ringel}
Let G be a connected graph with $p$ vertices and $q$ edges, and let $\rho$ be a maximal rotation of G. Then the genus of $G$ is $g$, where $p-q+r(\rho)=2-2g$.
\end{theorem}

As a second example, consider $G(1,3,3)$. See Figure \ref{fig:G(1,3,3)_hex}. Here several vertices and edges are drawn twice; the reader should view identical objects as being identified. We color the ``a'', ``b'', and ``c'' edges red, green, and blue, respectively. To simplify graph labeling, we write $x=ba$. Here, we have that $ca=ab$ and also that $cb=abab$. Note that, since $cba$ is a Euclidean reflection, $cbacba=(cba)^2=e$. So following edges labeled a,b,c,a,b,c in that order will yield a length 6 circuit, no matter where we start in the graph. These circuits are indicated in Figure \ref{fig:G(1,3,3)_hex} by the circular arrows within each hexagon. Observe that each edge in $G(1,3,3)$ is traversed exactly twice by these circuits -- once in each direction.

Next, consider Figure \ref{fig:G(1,3,3)_torus}, wherein we have taken the graph in Figure \ref{fig:G(1,3,3)_hex} and redrawn it slightly so that the outer boundary forms a larger hexagon whose opposite sides are identified. Since it is well-known that identifying opposite sides of a hexagon creates a genus 1 surface, we see that $G(1,3,3)$ has genus at most 1.

\begin{figure}[h]
    \centering
    \includegraphics[scale=.25]{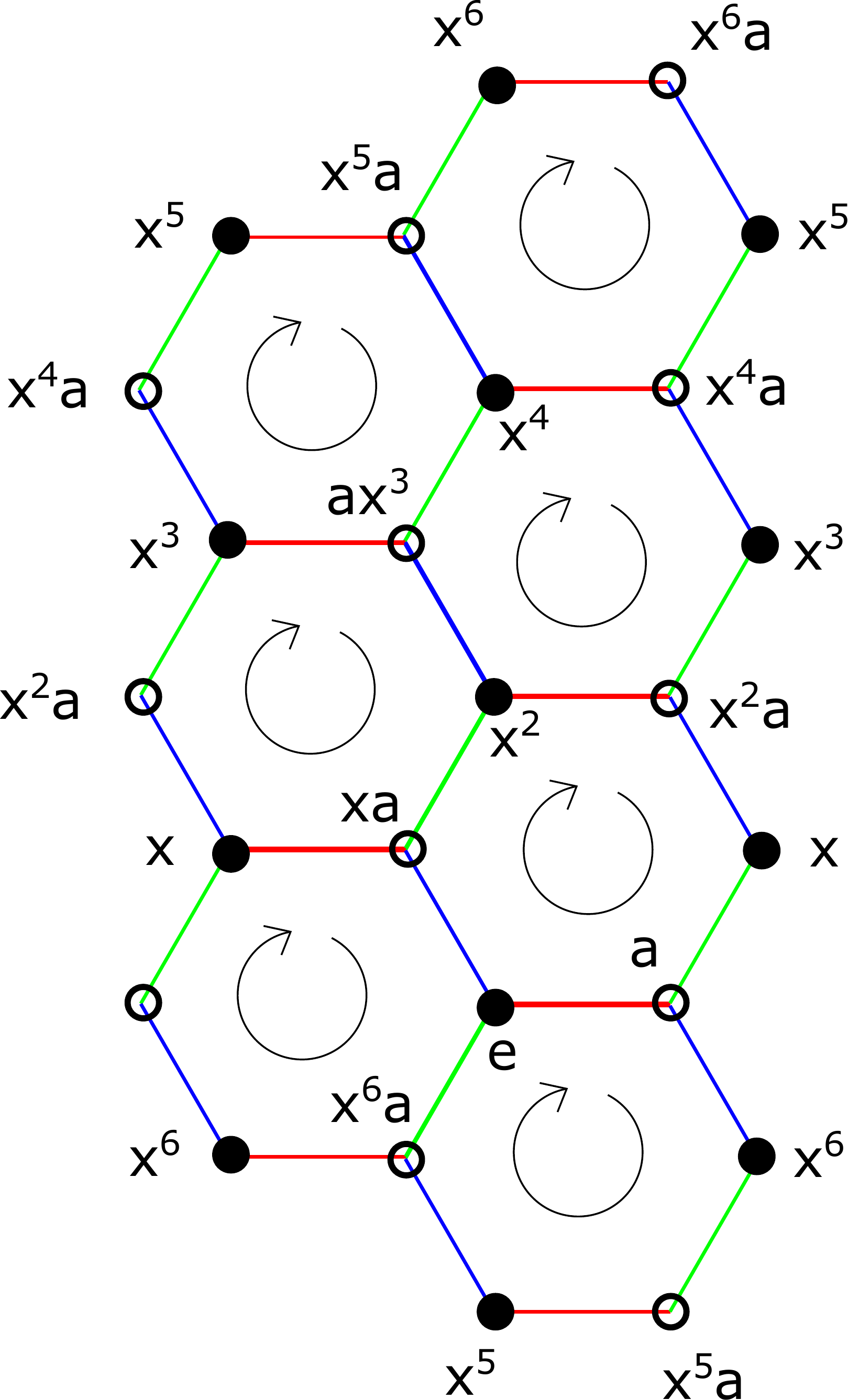}
    \caption{G(1,3,3), with ``a'' ``b'' and ``c'' edges colored red, green and blue respectively. Some edges and vertices are drawn twice. Circuits are indicated with circular arrows. We define $x=ba$.}
    \label{fig:G(1,3,3)_hex}
\end{figure}

\begin{figure}[h]
    \centering
    \includegraphics[scale=.2]{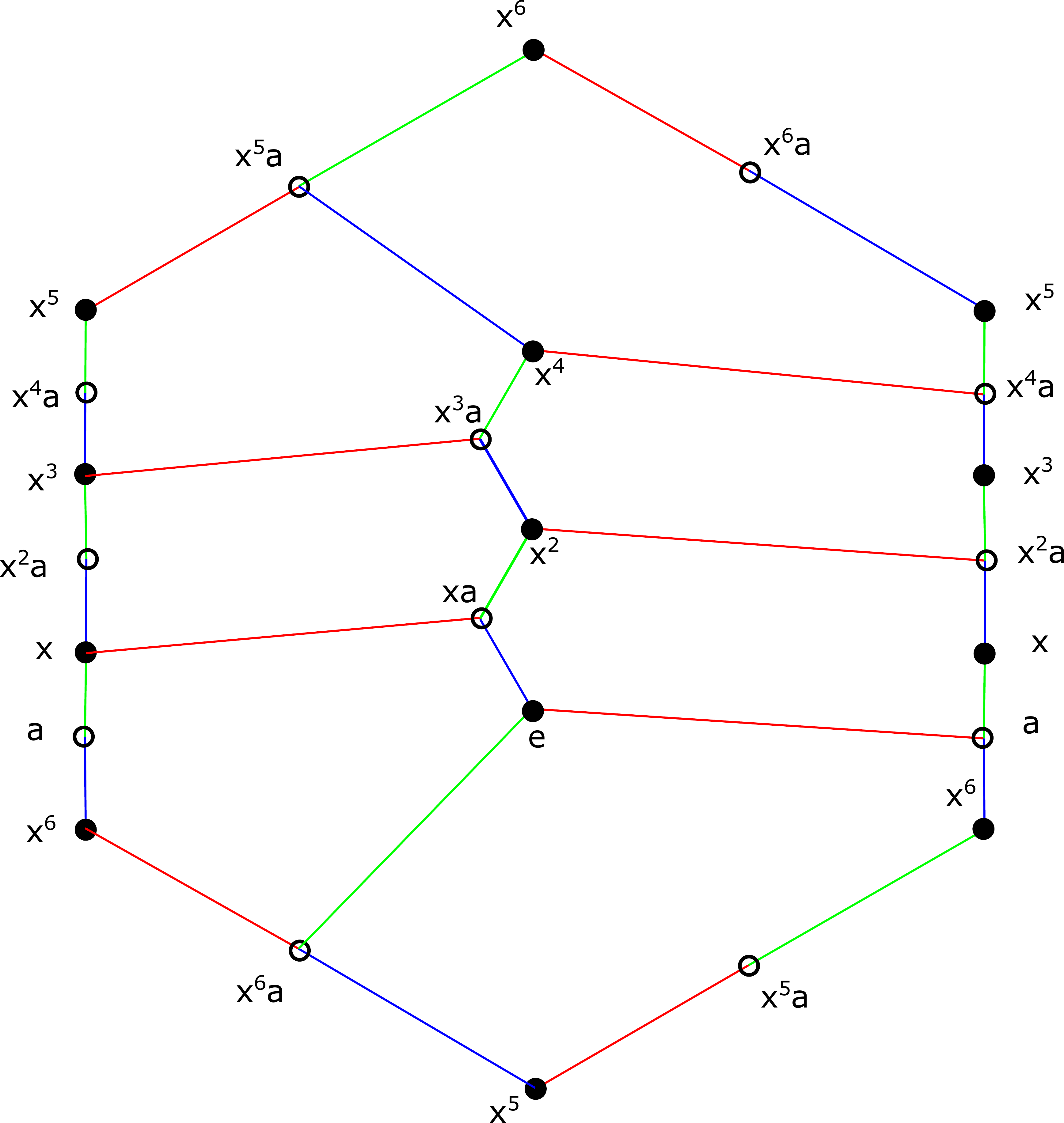}
    \caption{G(1,3,3), drawn on a hexagonal torus.}
    \label{fig:G(1,3,3)_torus}
\end{figure}

\section{Results}

The goal of this section is Theorem \ref{main_theorem}, which states that the genus of $G(p_1,p_2,p_3)$ is always zero or one. This result is in contrast to the fact that the surface $X(p_1,p_2,p_3)$ can have arbitrarily high genus \cite{Aurell_Itzykson}.

We begin by generalizing the last example in the previous section.

\begin{lemma}\label{lemma:max_genus}
The Cayley graph $G=G(p_1,p_2,p_3)$ has genus at most 1.
\end{lemma}

\begin{proof}
Let $n=p_1+p_2+p_3$. Since $D_n$ has $2n$ elements, $G$ has $2n$ vertices. Since $G$ is 3-regular, it follows that $G$ has $\frac{3}{2}(2n)=3n$ edges. Let $\rho$ be a maximal rotation of $G$. By Theorem \ref{thm:rotation_genus}, $2n-3n+r(\rho)=2-2g$, so $g=1+\dfrac{n-r(\rho)}{2}$.

Now consider the graph rotation $\rho_1$ defined on $G$ in the following way. Each vertex $x$ of $G$ is adjacent to three other vertices $v_{x,a},v_{x,b},v_{x,c}$ via edges with labels $a$, $b$, and $c$, respectively. Using this notation, for each $x$, define the rotation $\rho_1$ at $x$ to be $x.v_{x,a}v_{x,b}v_{x,c}$. Then any walk whose first edge is $a$ will begin with the edge string $abcabc$. Since the element $abc\in D_n$ is a reflection, we know that it is its own inverse. Hence $abcabc$ is the identity in $D_n$, from which it follows that $abcabc$ always describes a circuit in the directed graph. Indeed, any directed edge in $G$ is a part of such a circuit. Since $G$ has $6n$ directed edges, it follows that $\rho_1$ induces exactly $r(\rho_1)=\dfrac{6n}{6}=n$ circuits. Thus $g\leq 1+\dfrac{n-r(\rho_1)}{2}=1$. That is, the genus of $G$ is at most 1.
\end{proof}

\begin{figure}[h]
    \centering
    \includegraphics[scale=.15]{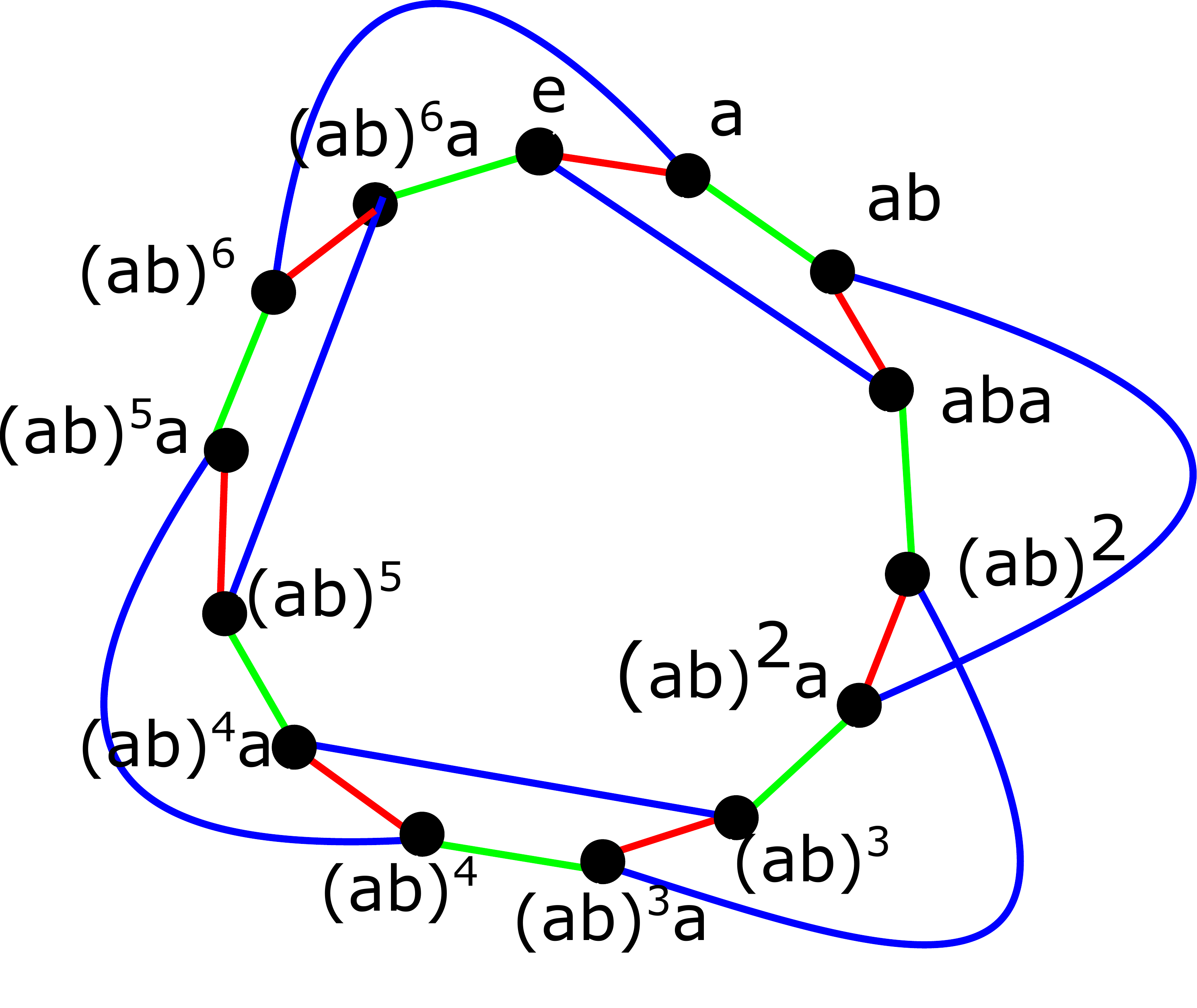}
    \caption{A drawing of $Cay(\{a,b,c\}, D_{7})$.}
    \label{fig:CayD7}
\end{figure}

\begin{figure}[h]
    \centering
    \includegraphics[scale=.15]{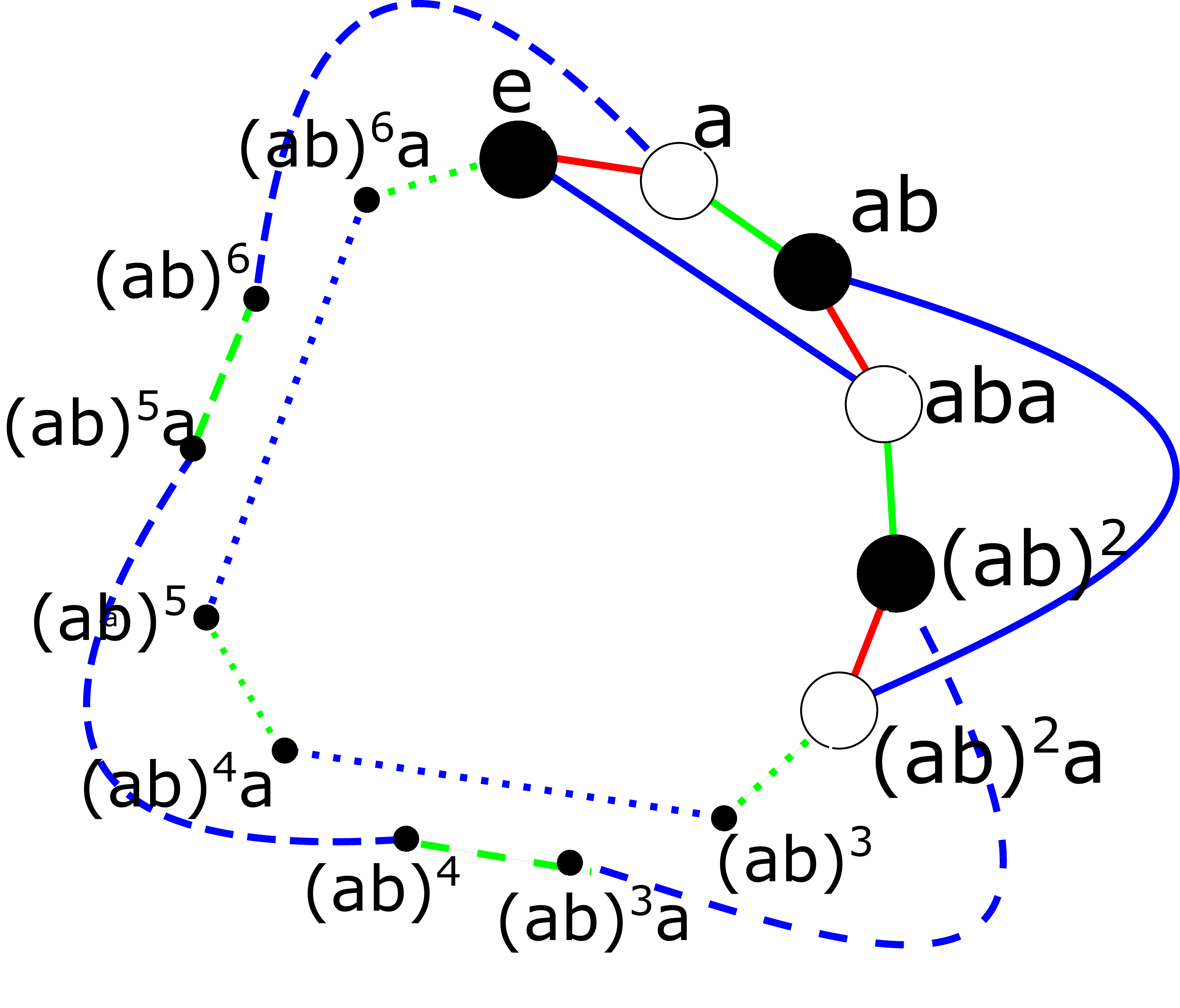}
    \caption{A subgraph of the Cayley graph for an isosceles triangle when $n=7$, pictured as a subdivision of $K_{3,3}$. The dashed edges form a subdivision of a single edge of $K_{3,3}$, as do the dotted edges. }
    \label{fig:K33subdivision}
\end{figure}

\begin{figure}[h]
    \centering
    \includegraphics[scale=.15]{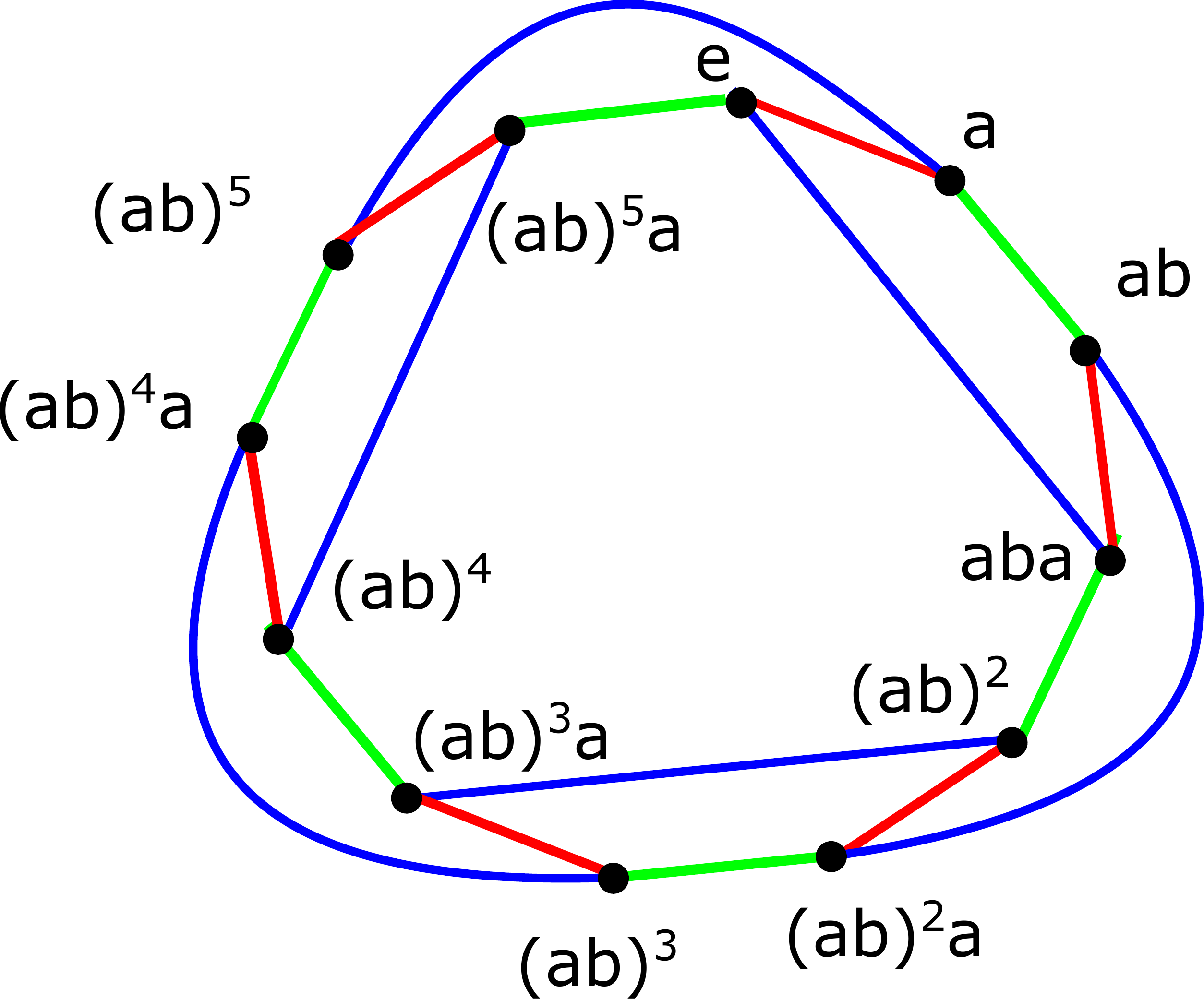}
    \caption{The Cayley graph for an isosceles triangle when $n=6$.}
    \label{fig:CayD6}
\end{figure}

\begin{lemma}\label{lemma:isosceles_four_word}
$G=G(p_1,p_2,p_3)$ has a length four circuit if and only if $T=T(p_1,p_2,p_3)$ is either isosceles or a right triangle.
\end{lemma}
\begin{proof}
Suppose that $G(p_1,p_2,p_3)$ has a length four circuit. Then without loss of generality we may say that the circuit is either $abab$ or $abac$, using the notation of Figure \ref{fig:reflections}. If $abab$ is a circuit in $G$ then $abab$ is the identity in $\Gamma=\Gamma(p_1,p_2,p_3)$. Since $ab\in\Gamma$ is the Euclidean rotation by $2\alpha_3$, we then have that $4\alpha_3=2\pi$, so that $\alpha_3=\dfrac{\pi}{2}$, and we see that $T$ is a right triangle.

If $abac$ is a circuit in $G$ then $abac$ is the identity in $\Gamma$. Since $ac\in\Gamma$ is the Euclidean rotation by $-2\alpha_2$, this implies that $2\alpha_3-2\alpha_2=0$; hence $\alpha_2=\alpha_3$ and $T$ is isosceles.

Conversely, suppose that $T(p_1,p_2,p_3)$ is isosceles. Without loss of generality say that $p_2=p_3$. Then we see that $abac$ is the identity in $\Gamma$, so $abac$ is a length 4 circuit in $G$. Finally, suppose instead that $T(p_1,p_2,p_3)$ is a right triangle. Say $\alpha_3=\dfrac{\pi}{2}$. Then it follows that $abab\in\Gamma$ is the identity and hence $abab$ is a circuit in $G$.   
\end{proof}

\begin{lemma}\label{lemma_two_generators}
For isosceles and right triangles, two of the three reflections generate $D_n$. 
\end{lemma}
\begin{proof}
Suppose that $T=T(p_1,p_2,p_3)$, with $n=p_1+P_2+p_3$. We have that $ab=Rot\left(\dfrac{2p_3\pi}{n}\right)$. The order of $ab$ is the smallest positive integer $k$ such that: $$(ab)^k=Rot(0)$$ $$Rot\left(\dfrac{2p_3k\pi}{n}\right)=Rot(0)$$ $$p_3k\equiv 0 \pmod n.$$ This implies that if $\gcd(p_3,n)=1$ then the order of $ab$ is $n$. In turn, if $ab$ has order $n$ then it generates an index 2 subgroup of $D_n$, and it follows that $a$ and $b$ together generate all of $D_n$.

Hence to prove the lemma it suffices to show that one of the $p_i$ is relatively prime to $n$. First suppose that $T$ is isosceles. Without loss of generality say that $p_2=p_3$. Let $k=\gcd(n,p_3)$. Then since $p_1=n-2p_3$ we see that $k$ divides $p_1$. But $\gcd(p_1,p_2,p_3)=1$, so $k=1$. Therefore $p_3$ and $n$ are relatively prime.  

Now suppose instead that $T$ is a right triangle. Without loss of generality, say that $\alpha_2$ is the right angle, so that $p_2=\dfrac{n}{2}$ and $p_1+p_3=p_2$. Note that $p_3$ and $p_1$ cannot both be even, because if they were then $p_2$ would also be even, but $\gcd(p_1,p_2,p_3)=1$. So let $p_3$ be odd and write $k=\gcd(n,p_3)$. Since $p_3$ is odd, $k$ is odd. Note that $2p_1=n-2p_3$, so since $k$ is odd, $k$ divides $p_1$ by Euclid's Lemma. But then $k$ also divides $p_2=p_1+p_3$, and hence $k=1$ since $\gcd(p_1,p_2,p_3)=1$. Thus $p_3$ is relatively prime to $n$. 
\end{proof}

Note that it is not the case for \emph{all} rational triangles that at least one of the $p_i$ is relatively prime to $n$. For example, consider $T(5,9,16)$.

\begin{lemma}\label{lemma:isosceles_genus}
Let $G=G(p_1,p_2,p_3)$ be the Cayley graph for an isosceles triangle with interior angles $\dfrac{p_1\pi}{n},\dfrac{p_2\pi}{n},\dfrac{p_3\pi}{n}$. Then the genus of $G$ is $0$ if and only if $n$ is even.
\end{lemma}
\begin{proof}
Suppose that $n$ is odd. We will show that the corresponding graph is not planar. Without loss of generality, we label our isosceles triangle as in Figure \ref{fig:reflections} so that $\alpha_2=\alpha_3$. Then $c=aba$ and we have a Cayley graph such as the one depicted in Figure \ref{fig:CayD7}. Note that Lemma \ref{lemma_two_generators} guarantees that the subgraph $Cay(\{a,b\},D_n)$ consists of a single cycle. We choose two sets of three vertices: $V_1=\{e,ba,(ba)^2\}$ and $V_2=\{a,aba,a(ba)^2\}$ to be our bipartite sets. We already have edges connecting $ba$ to each element of $V_2$; we also have edges connecting $e$ to $a$ and $aba$; and we have edges connecting $(ba)^2$ to $aba$ and $a(ba)^2$. This leaves two edges to complete a subdivision of $K_{3,3}$. As exemplified in Figure $\ref{fig:K33subdivision}$, we can connect $a$ to $(ba)^2$ via the sequence $a,(ba)^{n-1},a(ba)^{n-2},(ba)^{n-3},a(ba)^{n-4},\ldots,a(ba)^3,(ba)^2$. Similarly, we can connect $e$ to $a(ba)^2$ via the sequence $e,a(ba)^{n-1},a(ba)^{n-2},a(ba)^{n-3},(ba)^{n-4},\ldots,(ba)^3,a(ba)^2$. Hence by Kuratowski's Theorem, the graph is not planar; that is, its genus is not zero.

Now suppose that $n$ is even. We will show that the graph is planar. Since $a$ and $b$ generate $D_n$, we can draw a loop in the plane with all vertices of $D_n$ on it, using all the $a$ and $b$ edges. It remains to draw all the $c$ edges without creating any edge crossing. As exemplified in Figure \ref{fig:CayD6}, this can be done by drawing all the $c$ edges connecting $(ba)^{2k}$ to $a(ba)^{2k+1}$ inside the loop and drawing all the $c$ edges connecting $(ba)^{2k+1}$ to $a(ba)^{2k+2}$ outside the loop. Thus, the graph is planar; that is, its genus is zero.

\end{proof}
\begin{lemma}
If $T(p_1,p_2,p_3)$ is a right triangle then $G=G(p_1,p_2,p_3)$ is not planar.
\label{lemma:right_nonplanar}
\end{lemma}
\begin{proof}
As in the previous lemma we will demonstrate a subgraph of $G=G(p_1,p_2,p_3)$ isomorphic to a subdivision of $K_{3,3}$. Since $T$ is a right triangle, let $\alpha_3=\dfrac{\pi}{2}$. Thus $p_3=p_1+p_2$, and so $n=p_1+p_2+p_3=2(p_1+p_2)$ is even. Since $\gcd(p_1,p_2,p_3)=1$, at least one of $p_1$ and $p_2$ is odd. Hence we may let $p_1$ be odd. It follows that $\gcd(p_1,n)=1$. Therefore $k=\dfrac{n}{2}$ is the smallest positive integer satisfying 
$$2kp_1\equiv 0 \pmod n$$
$$2kp_1\equiv n \pmod n$$
$$kp_1\equiv \dfrac{n}{2} \pmod n$$
$$Ref\left(\dfrac{kp_1\pi}{n}\right)=Ref\left(\dfrac{(n/2)\pi}{n}\right)$$
$$c=a(ba)^k$$

Again, Lemma \ref{lemma_two_generators} guarantees that the subgraph $Cay(\{a,b\},D_n)$ consists of a single cycle. Consider Figure \ref{fig:right_subdivision}, which demonstrates that for $n>4$ the sets $V_1=\{e,(ba)^2,(ba)^{n/2+1}\}$ and $V_2=\{aba,a(ba)^{n/2},a(ba)^{n/2+2}\}$ can be used as the bipartite sets for a subgraph of$G$ that is a subdivision of $K_{3,3}$. This diagram suffices for $n>4$; if $n=4$ then $T=T(2,1,1)$ is isosceles so the claim follows from Lemma \ref{lemma:isosceles_genus}.
\end{proof}

\begin{figure}[h]
    \centering
    \includegraphics[scale=.4]{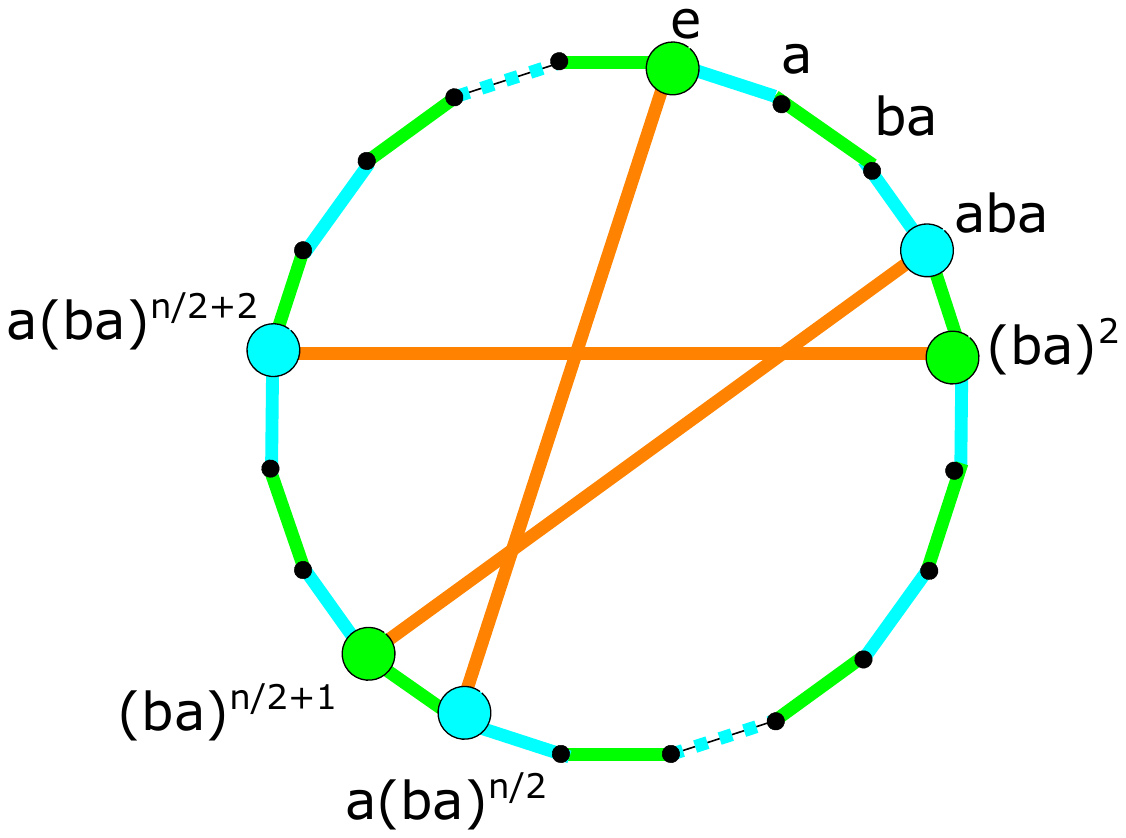}
    \caption{A subgraph of the Cayley graph for a right triangle which is isomorphic to a subdivision of $K_{3,3}$. The three blue vertices form a bipartite set; the three green blue vertices form the other.}
    \label{fig:right_subdivision}
\end{figure}

\begin{theorem}\label{main_theorem}
The genus of $G=G(p_1,p_2,p_3)$ is always 0 or 1. In particular, the genus is zero if and only if $T(p_1,p_2,p_3)$ is isosceles and $n=p_1+p_2+p_3$ is even.
\end{theorem}
\begin{proof}
By Lemma \ref{lemma:max_genus}, the genus $g$ of $G(p_1,p_2,p_3)$ is at most 1. Next we determine when $g=0$. Suppose that $g=0$, and that $\rho$ is a maximal rotation of $G$. Since all circuits of $G$ have even length of at least 4, it follows that $\rho$ must induce a circuit of length 4. By Lemma \ref{lemma:isosceles_four_word}, this can only occur if $T(p_1,p_2,p_3)$ is right or isosceles. But by Lemma \ref{lemma:right_nonplanar}, a genus 0 graph cannot arise from a right triangle. Lemma \ref{lemma:isosceles_genus} finishes the proof.
\end{proof}

We observe that Theorem \ref{main_theorem} easily extends to any polygon whose sides have no more than 3 distinct slopes.

\begin{corollary}
Let $P$ be a rational-angled polygon in the plane and let $S$ be the set of slopes (possibly including $\infty$) of $P$. If $|S|\leq 3$ then the Cayley graph arising from $P$ will have genus 0 or 1.
\end{corollary}
\begin{proof}
The Cayley graph $G$ arising from $P$ is actually $m$-regular, where $m=|S|$. If $m=3$ then $G$ is isomorphic to the graph of some triangle with the same set of slopes; hence Theorem \ref{main_theorem} applies. If $m=2$ then $G$ is a connected 2-regular graph and hence a cycle, which has genus 0. 
\end{proof}

So, for example, all trapezoids will yield graphs of genus 0 or 1, as will all ``L-shaped'' tables.
\section{Future Work}
It would be interesting to extend the results of this paper to all rational-angled polygons. One easily computes from Theorem \ref{thm:rotation_genus} that polygons with more than three edge slopes will correspond to graphs which have genus at least 2; however, we do not have an exact formula in the general case. 

\medskip

\printbibliography[
heading=bibintoc,
title={Bibliography}
] 
\end{document}